\documentclass[a4paper]{amsart}

\usepackage{amsmath,amstext,amssymb,mathrsfs,amscd,amsthm,indentfirst, bbm}
\usepackage{amsfonts}
\usepackage{tikz, tikz-cd}
\usetikzlibrary{matrix,calc,positioning,arrows,decorations.pathreplacing,decorations.markings,patterns}
\tikzset{->-/.style={decoration={
			markings,
			mark=at position #1 with {\arrow{>}}},postaction={decorate}}}
\tikzset{-<-/.style={decoration={
					markings,
					mark=at position #1 with {\arrow{<}}},postaction={decorate}}}

\newcommand\mnote[1]{\marginpar{\tiny #1}}

\usepackage{booktabs}
\usepackage{verbatim}
\usepackage{enumerate}

\usepackage{graphicx}

\newcommand{\bC}{\mathbb{C}}

\newcommand{\bF}{\mathbb{F}}

\newcommand{\bQ}{\mathbb{Q}}
\newcommand{\bR}{\mathbb{R}}

\newcommand{\bZ}{\mathbb{Z}}

\newcommand{\cL}{\mathcal{L}}

\newcommand\lra{\longrightarrow}

\newcommand\Diff{\mathrm{Diff}}

\newcommand\colim{\operatorname*{colim}}

\newcommand{\R}{\bR}

\renewcommand{\epsilon}{\varepsilon}

\newcommand{\SO}{\mathrm{SO}}
\newcommand{\OO}{\mathrm{O}}
\newcommand{\UU}{\mathrm{U}}

\newcommand{\Top}{\mathrm{Top}}

\newcommand{\STop}{\mathrm{STop}}

\newcommand{\Sign}{\mathrm{Sign}}
\newcommand{\ch}{\mathrm{ch}}

\mathchardef\ordinarycolon\mathcode`\:
\mathcode`\:=\string"8000
\begingroup \catcode`\:=\active
  \gdef:{\mathrel{\mathop\ordinarycolon}}
\endgroup

\usepackage{amsthm}
\theoremstyle{plain}

\newtheorem{theorem}{Theorem}[section]

\newtheorem{proposition}[theorem]{Proposition}
\newtheorem{lemma}[theorem]{Lemma}

\newtheorem{corollary}[theorem]{Corollary}

\theoremstyle{definition}
\newtheorem{definition}[theorem]{Definition}

\theoremstyle{remark}

\newtheorem{remark}[theorem]{Remark}
\newtheorem*{remark*}{Remark}

\numberwithin{equation}{section}

\usepackage[hypertexnames=false]{hyperref}

\title[Algebraic independence of topological Pontryagin classes]{Algebraic independence of\\topological Pontryagin classes}

\author{S{\o}ren Galatius}
\email{galatius@math.ku.dk}
\address{Department of Mathematics\\
	University of Copenhagen\\
	Denmark}

\author{Oscar Randal-Williams}
\email{o.randal-williams@dpmms.cam.ac.uk}
\address{Centre for Mathematical Sciences\\
Wilberforce Road\\
Cambridge CB3 0WB\\
UK}

\date{\today}
\begin{document}
\begin{abstract}
We show that the topological Pontryagin classes are algebraically independent in the rationalised cohomology of $B\Top(d)$ for all $d \geq 4$.
\end{abstract}
\maketitle

\section{Statement}

The spaces $B\OO(d)$ classifying $d$-dimensional vector bundles carry two well-known kinds of rational characteristic classes. On one hand, for $d=2n$ there is the Euler class $e \in H^{2n}(B\OO(2n);\bQ^{w_1})$, in cohomology with twisted coefficients corresponding to the determinant line: its square yields an untwisted cohomology class $e^2 \in H^{4n}(B\OO(2n);\bQ)$. On the other hand, there are the Pontryagin classes $p_i \in H^{4i}(B\OO(d);\bQ)$, which are stable in the sense that they extend to $B\OO = \colim_d B\OO(d)$. By their construction these satisfy some elementary relations: 
\begin{align}\tag{$\dagger$}\label{rels}
\begin{split}
p_i &= 0  \text{ for } d < 2i,\\
p_n &= e^2  \text{ for } d=2n.
\end{split}
\end{align}
Furthermore, the Euler and Pontryagin classes give a complete description of the cohomology of $B\OO(d)$, and these relations are the only ones satisfied: we have
$$H^*(B\OO(d);\bQ) = \begin{cases}
\bQ[p_1, p_2, \ldots, p_n] & d=2n+1\\
\bQ[p_1, p_2, \ldots, p_n, e^2]/(p_n-e^2) & d=2n.
\end{cases}$$

There are analogous spaces $B\Top(d)$ classifying fibre bundles with fibre the euclidean space $\bR^d$, and structure group $\Top(d)$, the group of homeomorphisms of $\bR^d$ fixing the origin. Neglecting the fibrewise linear structure of a vector bundle gives maps $B\OO(d) \to B\Top(d)$, and the induced map of stabilisations $B\OO \to B\Top$ is a rational homotopy equivalence (combine \cite[Essay V Theorem 5.5]{kirbysiebenmann} with \cite{kervairemilnor}).
We therefore have an isomorphism
$$H^*(B\Top;\bQ) \overset{\sim}\longrightarrow H^*(B\OO;\bQ) = \bQ[p_1, p_2, p_3, \ldots].$$
This means that there are uniquely defined rational cohomology classes on $B\Top$ which restrict to the usual Pontryagin classes on $B\OO$: these are the topological Pontryagin classes. They can be pulled back to $B\Top(d)$ for any $d$. The (squared) Euler class is defined for an $S^{2n-1}$-fibration, so can also be constructed on $B\Top(2n)$ by considering $E\Top(2n) \times_{\Top(2n)}(\bR^{2n} \setminus 0) \to B\Top(2n)$.

As their construction on $B\Top(d)$ is indirect it is by no means clear whether the relations \eqref{rels} should be expected to hold here (cf.\ \cite[p.\ 210]{Su2}), and Weiss \cite{weissdalian} has recently proved the rather surprising fact that they do not. Our goal here is to show that in fact no equations hold among these classes, using very different methods to Weiss.

\begin{theorem}\label{thm:Even}
  For $2n \geq 6$ the maps
\begin{equation*}
\bQ[e^2, p_1, p_2, \ldots] \lra H^*(B\Top(2n);\bQ) 
\end{equation*}
are injective. 
\end{theorem}

There is an immediate consequence in odd dimensions. There is a certain characteristic class $E$ of $S^{2n}$-fibrations which restricts to $e^2$ when evaluated on the fibrewise unreduced suspension of a $S^{2n-1}$-fibration (see \cite[Section 1.2.2]{KrannichR-WOdd} for more on this class). It can be evaluated on the underlying $S^{2n}$-fibration of the universal bundle over $B\Top(2n+1)$, and Theorem \ref{thm:Even} then implies: 

\begin{corollary}\label{cor:Odd}
For $2n+1 \geq 7$ the map
\begin{equation*}
\bQ[E, p_1, p_2, \ldots] \lra H^*(B\Top(2n+1);\bQ)
\end{equation*}
is injective. 
\end{corollary}

\begin{remark}
Weiss' result (and those of \cite{KR-WDisks} and \cite{KrannichR-WOdd} in this direction) is weaker in that it shows that $p_i \neq 0$ on $B\Top(d)$ for $i < f(d)$ with $f(d)-2d \to \infty$ (rather than all $p_i$ being nonzero, and even algebraically independent, on a single $B\Top(d)$), but is stronger in that it shows that these classes evaluate nontrivially against the image of the Hurewicz map. Krannich and Kupers have announced the result that all $p_i$ evaluate nontrivially against the image of the Hurewicz map on $B\Top(d)$ for any $d \geq 6$. This property does not seem to be accessible by the methods we use here.
\end{remark}

Identifying $D^{2n}$ with the cone on $S^{2n-1}$ yields continuous homomorphisms
\begin{equation}\label{eq:17}
  \Diff(S^{2n-1}) \lra \Top(S^{2n-1}) \lra \Top(D^{2n}) \lra \Top(2n),
\end{equation}
so we may pull back the (squared) Euler and Pontryagin classes to $\Diff(S^{2n-1})$.  As a byproduct of our proof of the above results, we obtain analogous results for the classifying spaces of any of these four groups, for example:
\begin{theorem}\label{thm:diff-spheres}
  For $2n-1 \geq 9$ the homomorphism
  \begin{equation*}
    \bQ[e^2, p_1, p_2, \dots] \lra H^*(B\Diff(S^{2n-1});\bQ)
  \end{equation*}
  is injective.
\end{theorem}

\subsection{Strategy}
\label{sec:strategy}

To show that a non-zero abstract polynomial $\Xi \in \bQ[e^2, p_1, p_2, \dots]$ defines a non-zero class in $H^*(B\Top(2n);\bQ)$, it is convenient to rewrite it in terms of the geometrically more significant Hirzebruch $L$-classes $\cL_i$, which are related to the Pontryagin classes by an invertible polynomial transformation.  As we explain in Section~\ref{sec:recoll-mathc-class} below, they admit preferred lifts to classes
\begin{equation*}
  \cL_i \in H^{4i}(B\Top(2n);\bZ[\tfrac1{(2i+1)!}]).
\end{equation*}
By a recent theorem of Kupers \cite{kupersdisk} the space $B\Top(2n)$ has finite type for $2n \geq 6$, so to show that a non-zero abstract polynomial $\Xi \in \bQ[e^2, \cL_1, \dots, \cL_m]$ defines a non-zero class in $H^*(B\Top(2n);\bQ)$, it suffices to
\begin{enumerate}[(i)]
\item choose a $\bZ[\frac1{N!}]$-integral lift of $\Xi$ for some $N \geq 2m+1$,
\item prove that the image of the lift of $\Xi$ is non-zero in $H^*(B\Top(2n);\bF_p)$ for infinitely many primes $p > N$.
\end{enumerate}
Similarly for $B\Diff(S^{2n-1})$, which Kupers has also shown to have finite type as long as $2n-1 \geq 9$.

This is what we will do, by constructing homomorphisms $\phi: C_p \to \Diff(S^{2n-1})$ adapted to the polynomial $\Xi$ to show that $(B\phi)^*\Xi \neq 0 \in H^*(BC_p;\bF_p)$.

\subsection{Finiteness and low dimensional cases}
\label{sec:finit-low-dimens}

In the form stated above, the proofs of Theorems~\ref{thm:Even} and~\ref{thm:diff-spheres} depend on Kupers' finiteness results.  In turn, Kupers' work relies on ideas of Weiss \cite{weissdalian}, in particular the combination of embedding calculus with the authors' previous work on moduli spaces of manifolds. This dependence may be avoided however, by replacing rational cohomology by rationalised integral cohomology $X \mapsto H^*(X;\bZ) \otimes \bQ$. To prove that a class $x \in H^{2r}(X;\bZ) \otimes \bQ$ is non-zero, it suffices---whether or not the space $X$ is known to have finite type---to choose a lift $x' \in H^{2r}(X;\bZ) \otimes \bZ[\tfrac1{N!}]$ for some $N$ and construct maps $j_p: BC_p \to X$ with $j_p^*(x') \neq 0 \in H^{2r}(BC_p;\bZ) \otimes \bZ[\tfrac1{N!}]$ for infinitely many primes $p > N$. Using this strategy with $X = B\Top(2n)$ or $X = B\Diff(S^{2n-1})$ the finiteness question may be circumvented, and we obtain the following.
\begin{theorem}\label{thm:low-dimensional}
  The homomorphisms
  \begin{align*}
    \bQ[e^2, p_1, p_2, \dots] &\lra H^*(B\Top(2n);\bZ) \otimes \bQ & 2n \geq 4\\
    \bQ[E, p_1, p_2, \dots] &\lra H^*(B\Top(2n+1);\bZ) \otimes \bQ & 2n+1 \geq 5\\
    \bQ[e^2, p_1, p_2, \dots] &\lra H^*(B\Diff(S^{2n-1});\bZ) \otimes \bQ & 2n-1 \geq 5
  \end{align*}
  are injective.
\end{theorem}

\begin{remark}\label{rem:WeirdCoefficients}
  For a different perspective on the use of finite cyclic groups to detect non-torsion classes, we can replace rationalised integral cohomology by certain other functors.  For any space $X$, there are natural homomorphisms
\begin{align}\label{eq:16}
\begin{split}
  H^j(X;\bZ) \otimes \bQ \lra \colim_{N \to \infty} H^j(X;\bZ[\tfrac1{N!}]) &\lra \colim_{N \to \infty} \prod_{p > N} H^j(X; \bF_p)
   \\ =  H^j(X;\prod_p \bF_p) \otimes_\bZ \bQ &\lra H^j(X;\prod_p \bF_p) \otimes_{(\prod \bF_p)} L,
\end{split}
\end{align}
where the subscript $p$ runs over prime numbers and $L = (\prod_p \bF_p)/\mathfrak{m}$ for any maximal ideal $\mathfrak{m}$ containing all additive torsion elements.  For any subset $P$ of the set of prime numbers, the characteristic function defines an idempotent $e_P \in \prod_p \bF_p$ whose image in $L$ is either 0 or 1, depending on whether $e_P \in \mathfrak{m}$ or $1-e_P \in \mathfrak{m}$.  If $P$ excludes only finitely many primes then $1-e_P \in \mathfrak{m}$, and for any infinite set $P$ of prime numbers we can choose $\mathfrak{m}$ to contain $1-e_P$.

The rightmost ``cohomology theory'' in \eqref{eq:16} has the property that the canonical maps $BC_p \to \bC P^\infty$ induce an isomorphism
\begin{align}\label{eq:19}
\begin{split}
  H^{2r}(\bC P^\infty;L) &= H^{2r}(\bC P^\infty;\prod_{p \in P} \bF_p) \otimes_{(\prod \bF_p)} L \\
& \quad\quad\quad\quad   \lra \bigg(\prod_{p \in P} H^{2r}(BC_p;\bF_p)\bigg) \otimes_{(\prod \bF_p)} L,
  \end{split}
\end{align}
for any set $P$ of prime numbers such that $1-e_P \in \mathfrak{m}$. A collection of maps $j_p: BC_p \to X$ for an infinite set of prime numbers $p$ can therefore be used as a substitute for a single map $j: \bC P^\infty \to X$, at least for cohomological purposes such as detecting cohomology classes. By the strategy we have explained, our algebraic independence results (Theorem~\ref{thm:low-dimensional} and Theorem~\ref{thm:discrete} below) also hold for $H^*(-;\prod_p \bF_p) \otimes_{(\prod \bF_p)} L$, as well as the other functors in~\eqref{eq:16}.  

Because of the homomorphisms~\eqref{eq:16}, algebraic independence for these functors is in principle a stronger statement than for rationalised integral cohomology, in the absense of finiteness results.
\end{remark}

The statement of Theorem~\ref{thm:low-dimensional} is more elementary than the more recognisable statements of Theorems~\ref{thm:Even} and~\ref{thm:diff-spheres}, in the sense that it
\begin{enumerate}[(i)]
\item holds in a wider range of dimensions;
\item will be given an ``elementary'' proof, not relying on techniques or results from \cite{weissdalian}, \cite{kupersdisk}, moduli spaces of manifolds, or embedding calculus;
\item is easily seen to be equivalent to Theorems~\ref{thm:Even} and~\ref{thm:diff-spheres} using Kupers' finiteness results, in the dimensions where they apply.
\end{enumerate}
Rationalised integral cohomology is slightly awkward in some ways though.  For example, while we prove that $e^2 - p_2 \neq 0 \in H^8(B\Top(4);\bZ) \otimes \bQ$, we cannot rule out that this cohomology class might nevertheless vanish universally when evaluated on $\bR^4$-bundles over closed $8$-manifolds.

The homomorphisms $\phi: C_p \to \Diff(S^{2n-1})$ that we use to detect cohomology classes are of course continuous also when $\Diff(S^{2n-1})$ is given the discrete topology.  The proof of Theorem~\ref{thm:low-dimensional} in fact shows the following stronger result.
\begin{theorem}\label{thm:discrete}
  Let $2n-1 \geq 5$ and let $\Diff^\delta(S^{2n-1})$ denote the diffeomorphism group regarded as a discrete group.  Then the homomorphism
  \begin{align*}
    \bQ[e^2, p_1, p_2, \dots] &\lra H^*(B\Diff^\delta(S^{2n-1});\bZ) \otimes \bQ
  \end{align*}
  is injective.
\end{theorem}

Finally, let us  remark that by well-known theorems of Smale \cite{smaledisk} and Hatcher \cite{hatcherdisk}, the maps $B\OO(d) \to B\Top(d)$ and $B\OO(d+1) \to B\Diff(S^d)$ are  equivalences for $d \leq 3$, so the analogous maps in Theorem~\ref{thm:low-dimensional} are \emph{not} injective in these dimensions.

\subsection{Acknowledgements}

ORW was supported by the ERC under the European Union's Horizon 2020 research and innovation programme (grant agreement No.\ 756444) and by a Philip Leverhulme Prize from the Leverhulme Trust. SG was supported
by the European Research Council (ERC) under the European Union's Horizon 2020 research
and innovation programme (grant agreement No. 682922), and by the
Danish National Research Foundation (DNRF151).

\section{Linear and non-linear representations}

It will be convenient to replace the Pontryagin classes $p_i$ by the more geometrically significant Hirzebruch $L$-classes.

\subsection{Recollections on $\mathcal{L}$-classes}
\label{sec:recoll-mathc-class}

The Hirzebruch class $\mathcal{L}$ is a characteristic class of real vector bundles with values in rational cohomology, uniquely characterised by the property $\mathcal{L}(V \oplus W) = \mathcal{L}(V) \cdot \mathcal{L}(W)$ for vector  bundles $V$ and $W$, together with the formula
\begin{equation*}
  \mathcal{L}(L)  = f(c_1(L))
\end{equation*}
when $L$ is (the 2-dimensional real vector bundle underlying) a complex line bundle, where $f$ is the (even) power series given by $f(t) = t/\tanh(t) \in \bQ[[t]]$. This is a rephrasing of Hirzebruch's definition \cite[\S 1.5]{HirzebruchBook} via multiplicative sequences. There is a universal class  
\begin{equation*}
  \mathcal{L} = \sum_{i = 0}^\infty \mathcal{L}_i \in H^*(B\OO;\bQ)
\end{equation*}
with the property that a closed oriented $4k$-manifold $M$ has its signature computed as $\sigma(M) = \int_M \cL_k(TM)$: this is Hirzebruch's signature theorem.

Reducing the power series $f(t) = t/\tanh(t)$ modulo $t^{2i+1}$ gives a unit of the ring $\bZ[\frac1{(2i+1)!}][[t]]/(t^{2i+1})$, because the same is true for $\cosh(t)$ and $t^{-1}\sinh(t)$ by inspection, and in particular the coefficient of $t^{2i}$ in $f(t)$ lies in $\bZ[\frac1{(2i+1)!}]$.  Using that the direct sum map $B\UU(1)^n \to B\OO(2n)$ induces an isomorphism from $H^*(B\OO(2n);\bZ[\tfrac12])$ onto the subring of $H^*(B\UU(1)^n;\bZ[\frac12]) = \bZ[\frac12][t_1, \dots, t_n]$ consisting of the symmetric polynomials in $t_1^2, \ldots, t_n^2$, we deduce that there are unique $\bZ[\tfrac{1}{(2i+1)!}]$-integral refinements
\begin{equation}\label{eq:BOrefinement}
\cL_i \in H^{4i}(B\OO ; \bZ[\tfrac{1}{(2i+1)!}]).
\end{equation}
This refinement, when evaluated on the real vector bundle underlying a complex line bundle $L$, is given by the coefficient of $t^{2i}$ in the power series $f(c_1(L) \cdot t)$.

The induced ring homomorphism
\begin{equation*}
  \bQ[\cL_1, \cL_2, \dots] \lra H^*(B\OO;\bQ)
\end{equation*}
is an isomorphism, because it has the form (see \cite[\S 1.5]{HirzebruchBook})
\begin{equation}\label{eq:10}
  \cL_i \longmapsto \frac{2^{2i}(2^{2i-1}-1) |B_{2i}|}{(2i)!} p_i + \text{polynomial in $p_j$'s with $j < i$},
\end{equation}
where $B_{2i} \in \bQ^\times$ is the Bernoulli number.
Therefore over $\bQ$ the Hirzebruch classes can be expressed as polynomials in the Pontryagin classes, and vice versa.

Since $B\OO \to B\Top$ is a rational equivalence, the classes $\cL_i$ are the restrictions of unique classes in the cohomology of $B\Top$, expressed in terms of the topological Pontryagin classes by the same polynomials. The classes $\mathcal{L}_i \in H^{4i}(B\Top;\bQ)$ admit canonical $\bZ[\frac1{(2i+1)!}]$-integral refinements as follows.  The composition of spectrum maps
\begin{equation*}
  M\STop \xrightarrow{\Delta_\Top} ko[\tfrac12] \xrightarrow{-\otimes\bC} ku[\tfrac12] \xrightarrow{\mathrm{ch}_{2i}} \Sigma^{4i} H\bZ[\tfrac1{(2i+1)!}],
\end{equation*}
where $\Delta_\Top$ is the Sullivan orientation (see \cite[Section 3.3]{RWSign}), corresponds under Thom isomorphism to a  class $\overline{\cL}_i \in H^{4i}(B\STop;\bZ[\frac1{(2i+1)!}])$.  Since $B\STop \to B\Top$ is a $\bZ[\frac12]$-equivalence and since $2i+1 > 2$, these classes extend uniquely to $H^{4i}(B\Top;\bZ[\frac1{(2i+1)!}])$.  The relationship to the Hirzebruch classes is
\begin{equation*}
  \cL_i(V) = \overline{\cL}_i(-V)
\end{equation*}
for virtual bundles $V$, and this formula provides a preferred $\bZ[\frac1{(2i+1)!}]$-integral refinement of $\cL_i$. Again, as  $H^{4i}(B\OO ; \bZ[\tfrac{1}{2}])$ is torsion-free these refinements pull back to  \eqref{eq:BOrefinement} on $B\OO$.

Finally, $B\Top$ has finite type, by the finiteness of $\pi_j(\Top/\OO)$ for all $j$ (by \cite[Essay V Theorem 5.5]{kirbysiebenmann} and \cite{kervairemilnor}) and the well known CW structure on $B\OO$.    Therefore the canonical homomorphism
\begin{equation*}
  H^{4i}(B\Top;\bZ) \otimes \bZ[\tfrac1{(2i+1)!}] \lra H^{4i}(B\Top;\bZ[\tfrac1{(2i+1)!}])
\end{equation*}
is an isomorphism, so $\cL_i$ lifts uniquely to a class which may be pulled back to
\begin{equation*}
  \cL_i \in H^{4i}(B\Top(d);\bZ) \otimes \bZ[\tfrac1{(2i+1)!}].
\end{equation*}

\subsection{$L$-classes of linear and non-linear representations}

As mentioned above, our main tool for studying the algebraic independence of $e^2, \cL_1, \cL_2, \ldots, \cL_k $ is to calculate their pullback along a large supply of maps $BC_p \to B\Top(2n)$ for varying primes $p > 2k+1$, where $C_p$ denotes the cyclic group of order $p$. We consider this group as the subgroup of $p$th roots of unity inside $\UU(1)$, and write $\chi : C_p \to \UU(1)$ for the inclusion considered as a 1-dimensional complex representation. We abbreviate the first Chern class of this representation as $c := c_1(\chi)$, so that 
$$H^*(BC_p ; \bZ) = \bZ[c]/(p \cdot c).$$
 To begin with, we have the following criterion for a tuple of classes in $H^*(BC_p;\bZ)$ to arise as the characteristic classes of \emph{linear} representations.
\begin{proposition}\label{prop:LinRep}
  Let $\rho: C_p \to \SO(2n)$ be a representation. Then there exist classes $a_1, \dots, a_n \in H^2(BC_p;\bZ)$ such that
  \begin{equation}
    \begin{aligned}
      e(\rho) & = \prod_{j = 1}^n a_j \in H^{2n}(BC_p;\bZ)\\
      \mathcal{L}_i(\rho) & = \ell_i(a_1, \dots, a_n) \in H^{4i}(BC_p; \bZ[\tfrac{1}{(2i+1)!}])
    \end{aligned}\label{eq:2}
  \end{equation}
  where $\ell_i$ are the symmetric polynomials characterised by
  \begin{equation*}
    \prod_{j=1}^n f(a_j t) = \sum_{i = 0}^\infty \ell_i(a_1, \dots, a_n) t^{2i},
  \end{equation*}
  with $f(t) = t/\tanh(t)$.

  Conversely, given any $a_1, \dots, a_n \in H^2(BC_p;\bZ)$, there exists a homomorphism $\rho: C_p \to \SO(2n)$ satisfying~\eqref{eq:2}.
\end{proposition}
\begin{proof}
  Given $\rho$, there exists an isomorphism of oriented representations
  \begin{equation*}
    \rho \cong \rho_1 \oplus \dots \oplus \rho_n
  \end{equation*}
  for homomorphisms $\rho_j: C_p \to \UU(1) = \SO(2)$. Setting $a_j := c_1(\rho_j)$, we have
  $$e(\rho) = \prod_{j=1}^n e(\rho_j) = \prod_{j = 1}^n a_j.$$
Furthermore, the $\bZ[\tfrac{1}{(2i+1)!}]$-integral refinement of $\cL_i$ evaluated on the sum $\rho_1 \oplus \dots \oplus \rho_n$ of complex line bundles is by definition $\ell_i(c_1(\rho_1), \dots, c_1(\rho_n)) = \ell_i(a_1, \ldots, a_n)$.
  
For the converse, the class $r \cdot c \in H^2(BC_p;\bZ)$ corresponds to the 1-dimensional complex representation $\chi^{r}$ given by the $r$th tensor power of $\chi$, so one takes the corresponding sum of such.
\end{proof}

Since we have $n$ degrees of freedom in choosing the $a_j$ we expect, roughly speaking, that for linear representations the classes $\cL_i$ with $i > n$ are expressible as polynomials in $\cL_1, \dots, \cL_n$. In contrast we shall show that for \emph{non-linear} representations $\rho: C_p \to \Top(2n)$ the $L$-classes are largely unconstrained, at least in a range of degrees.
The precise statement is as follows.
\begin{proposition}\label{prop:NonLinRep}
  Let $p$ be an odd prime and $n \geq 2$ an integer. Given any representation $\rho: C_p \to \SO(2n)$ having no trivial subrepresentations, and any tuple of elements
  \begin{equation*}
    x_i  \in H^{4i}(BC_p;\bZ), \quad\text{for $\lceil\tfrac n2\rceil \leq i \leq \tfrac{p-3}2$},
  \end{equation*}
  there exists a map $B\phi : BC_p \to B\STop(2n)$, such that
  \begin{align*}
    (B\phi)^*e &= e(\rho)  \in H^{2n}(BC_p ; \bZ),\\
    (B\phi)^*\cL_j &= \cL_j(\rho)  \in H^{4j}(BC_p ; \bZ[\tfrac{1}{(2j+1)!}]), & \text{for $1 \leq j < \lceil\tfrac n2\rceil$},\\
    (B\phi)^*\cL_j &= x_j  \in H^{4j}(BC_p ; \bZ[\tfrac{1}{(2j+1)!}]) & \text{for $\lceil\tfrac n2\rceil \leq j \leq \tfrac{p-3}2$}.
  \end{align*}

  For $n \geq 3$ we can furthermore arrange that this map $B\phi$ is induced by a homomorphism $\phi: C_p \to \STop(2n)$, which extends to the Pr\"ufer group $C_{p^\infty}$ and factors through a homomorphism $C_{p^\infty} \to \Diff^+(S^{2n-1})$ as in~\eqref{eq:17}.
\end{proposition}

\begin{remark}
  The homomorphisms $C_p \to \Diff^+(S^{2n-1})$ we will construct in the proof of this proposition will in fact define \emph{free} actions by diffeomorphisms.  The related question of free actions by homeomorphisms is discussed in \cite[Section 14E]{wallscm} by related methods.

  The extension to an action of the Pr\"ufer group is not strictly necessary for proving any of the results stated in the introduction, but may have independent interest and is proved in Section~\ref{sec:surg-exact-sequ}.  See also Remark~\ref{rem:prufer} below.
\end{remark}

\subsection{Proof of Theorem~\ref{thm:low-dimensional}, assuming Proposition~\ref{prop:NonLinRep}}

Let $P_i \in \bQ[x_1, \dots, x_i]$ be the unique polynomials such that
\begin{equation*}
  p_i = P_i(\cL_1, \dots, \cL_i) \in H^*(B\OO;\bQ) = H^*(B\Top;\bQ),
\end{equation*}
Write $k = \lceil \frac n2\rceil$ and consider the ring homomorphisms
\begin{equation}\label{eq:11}
  \bQ[e, p_1, p_2, \dots] \lra \bQ[e, x_1, x_2, \dots] \lra \bQ[a_1, \dots, a_n, x_k, x_{k+1}, \dots],
\end{equation}
where the first is the isomorphism given by $e \mapsto e$ and $p_i \mapsto P_i(x_1, \dots, x_i)$ and the second is given by
\begin{align*}
  e & \longmapsto \sigma_n(a_1, \dots, a_n) = a_1 \cdots a_n\\
  x_i &\longmapsto \ell_i(a_1, \dots, a_n) \quad \text{for $i < k$}\\
  x_i & \longmapsto x_i \quad \text{for $i \geq k$}.
\end{align*}
The symmetric polynomial $\ell_i = \ell_i(a_1, \dots, a_n)$ can be written in terms of elementary symmetric polynomials $\sigma_{2i}, \dots, \sigma_2$ by the same formula as~\eqref{eq:10} above, namely
\begin{equation*}
  \ell_i = \frac{2^{2i}(2^{2i-1}-1) |B_{2i}|}{(2i)!} \sigma_{2i} + \text{polynomial in $\sigma_{2j}$'s with $j < i$}.
\end{equation*}
Since $2k-2 < n$ and since the elementary symmetric polynomials $\sigma_1, \dots, \sigma_n$ are algebraically independent, the symmetric polynomials $\ell_1, \dots, \ell_{k-1}, \sigma_n$ are also algebraically independent, so the second homomorphism in~\eqref{eq:11} is injective.

Now let $\Xi$ be a non-zero element in the domain of the map
  \begin{equation}\label{eq:15}
    \bQ[e, p_1, p_2, \dots] \lra H^*(B\STop(2n);\bZ) \otimes \bQ
  \end{equation}
 of homogeneous degree $2r$.  Then the image of $\Xi$ under the composition~\eqref{eq:11} is also a non-zero polynomial, and by abuse of notation we write $\Xi(e, x_1, x_2, \dots,)$ and $\Xi(a_1, \dots, a_n, x_k, x_{k+1}, \dots)$ for the corresponding polynomials.  These only involve finitely many variables and finitely many denominators, so there are non-zero polynomials
\begin{align*}
  \Xi(e, x_1, x_2, \dots,x_m) &\in \bZ[\tfrac1{N!}][e, x_1, \dots, x_m]\\
  \Xi(a_1,\dots, a_n, x_k, \dots,x_m) &\in \bZ[\tfrac1{N!}][a_1, \dots, a_n,x_k, \dots, x_m]
\end{align*}
for some $m \geq k = \lceil\frac n2\rceil$ and some $N \geq 2m+1$.  The latter condition also implies that we have made unambiguous sense of the classes $\cL_1, \dots, \cL_m \in H^*(B\Top(2n);\bZ) \otimes \bZ[\frac1{N!}]$, and hence of the element
\begin{equation}\label{eq:12}
  \Xi(e, \cL_1, \dots, \cL_m) \in H^{2r}(B\STop(2n);\bZ) \otimes \bZ[\tfrac1{N!}].
\end{equation}
As explained in Section \ref{sec:finit-low-dimens}, to show that the image of $\Xi$ under~\eqref{eq:15} is non-zero, it suffices to show that the image of $\Xi(e, \cL_1, \dots, \cL_m)$ in $H^{2r}(B\STop(2n);\bF_p)$ is non-zero for infinitely many $p > N$.

We now claim that, after possibly increasing $N$, there exist units
\begin{equation*}
  \bar{a}_1,\ldots,\bar{a}_n, \bar{x}_k, \ldots, \bar{x}_m \in \bZ[\tfrac1{N!}]^\times
\end{equation*}
such that
\begin{equation*}
  \Xi(\bar{a}_1, \dots, \bar{a}_n, \bar{x}_k, \dots, \bar{x}_m) \in \bZ[\tfrac1{N!}]^\times.
\end{equation*}
To see this, let us temporarily use the same notation $\Xi: \bR^{n+m-(k-1)} \to \bR$ for the continuous map defined by the polynomial, which by our assumption that $\Xi \neq 0$ as an abstract polynomial is not the constant map zero.
The subset $(\bR\setminus\{0\})^{n+m-(k-1)} \subset \bR^{n+m-(k-1)}$ is open and dense, so its intersection with the non-empty open set $\Xi^{-1}(\R \setminus \{0\})$ remains non-empty and open.  By density of the rational numbers, we may choose an element
\begin{equation*}
  z = (\bar{a}_1, \dots, \bar{a}_n, \bar{x}_k, \dots, \bar{x}_m) \in (\bQ\setminus \{0\})^{n+m-(k-1)} \cap \Xi^{-1}(\R \setminus \{0\}).
\end{equation*}
After increasing $N$ if necessary, all coordinates of $z$, as well as the value $\Xi(z)$, will be units of $\bZ[\frac1{N!}]$.

For any prime $p > N$ we can then define $\rho := \chi^{\bar{a}_1} \oplus\dots \oplus \chi^{\bar{a}_n}: C_p \to \SO(2n)$. Applying Proposition~\ref{prop:NonLinRep} gives a map $B\phi: BC_p \to B\STop(2n)$, such that
\begin{align*}
  (B\phi)^*e &= \bar{a}_1 \cdots \bar{a}_n \cdot c^n\\
    (B\phi)^*\cL_i &= \ell_i(\bar{a}_i, \ldots, \bar{a}_n) \cdot c^{2i} \quad\text{for $1 \leq i < k$}\\
  (B\phi)^*\cL_i &= \bar{x}_i \cdot c^{2i} \quad\text{for $k \leq i \leq m$}
\end{align*}
in $H^*(BC_p ; \bZ[\tfrac{1}{N!}])$ and hence in $H^*(BC_p ; \bF_p)$, since $p > N$. Pulling back the class~\eqref{eq:12} along $B\phi$ we now get
\begin{multline*}
  \Xi(\bar{a}_1 c, \dots, \bar{a}_n c, \bar{x}_k c^{4k} , \dots, \bar{x}_m c^{4m}) = \Xi(\bar{a}_1, \dots, \bar{a}_n , \bar{x}_k, \dots, \bar{x}_m) c^{2r}\neq 0 \\ \in H^{2r}(BC_p;\bF_p).
\end{multline*}
This holds for all primes $p > N$, which as explained in Sections \ref{sec:strategy} and \ref{sec:finit-low-dimens} proves injectivity of the map~\eqref{eq:15}.

The statement about $B\Top(2n)$ in Theorem \ref{thm:low-dimensional} follows by taking invariants for the involution given by orientation reversal. The statement about $B\Top(2n+1)$ in Theorem \ref{thm:low-dimensional} follows by using the map $B\Top(2n) \to B\Top(2n+1)$, which pulls back the class $E$ to $e^2$. The final part of Theorem \ref{thm:low-dimensional} follows by applying the same argument to $B\Diff^+(S^{2n-1})$ for $n \geq 3$, in which case the detecting maps $B\phi$ factor through that space as asserted in the last part of Proposition~\ref{prop:NonLinRep}.
\qed

\begin{remark}\label{rem:prufer}
  For $n \geq 3$ it is possible to deduce non-vanishing in characteristic 0 from non-vanishing in characteristic $p$ using just a single prime $p > N$ instead of all $p > N$, in the notation of the above proof: any non-zero polynomial $\Xi$ is detected in $H^*(BC_{p^\infty};\bZ) \otimes \bQ \cong H^*(\bC P^\infty;\bQ_p)$ for some action of a Pr\"ufer group $C_{p^\infty}$ on $S^{2n-1}$.

Instead of the ``cohomology theory'' $X \mapsto H^*(X;\prod \bF_p) \otimes_{\prod \bF_p} L$ from Remark~\ref{rem:WeirdCoefficients} we can use $X \mapsto H^*(X;\bZ_p) \otimes_{\bZ_p} \bQ_p$, for which the canonical map $BC_{p^\infty} \to \bC P^\infty$ induces an isomorphism
    \begin{equation*}
      H^{2r}(\bC P^\infty;\bQ_p) = H^{2r}(\bC P^\infty;\bZ_p) \otimes_{\bZ_p} \bQ_p \lra H^{2r}(BC_{p^\infty};\bZ_p) \otimes_{\bZ_p} \bQ_p
    \end{equation*}
    analogous to~\eqref{eq:19}.  Since the map $H^{2r}(BC_{p^\infty};\bZ_p) \to H^{2r}(BC_{p^\infty};\bZ_p) \otimes \bQ_p$ can be identified with $\bZ_p \to \bQ_p$, a class in $H^{2r}(BC_{p^\infty};\bZ_p)$ has non-vanishing image in $H^{2r}(BC_{p^\infty};\bZ_p) \otimes \bQ_p \cong H^{2r}(\bC P^\infty;\bQ_p)$ provided it has non-vanishing image in $H^{2r}(BC_p;\bF_p)$, which is what we verified in the proof above.
  \end{remark}

\section{Signatures of topological manifolds}

In the next section we will use continuous actions of $C_p$ on closed oriented manifolds to prove Proposition~\ref{prop:NonLinRep}.  The $C_p$-manifolds that we will construct will have precisely two fixed points, $s$ and $t$, and give rise to an oriented fibre bundle
\begin{equation*}
  \pi: E \lra BC_p,
\end{equation*}
with two sections $s, t: BC_p \to E$ corresponding to the fixed points of the action.  The fiberwise tangent microbundle $T_\pi E$ is classified by a map $E \to B\STop(2n)$, and the composition
\begin{equation*}
  BC_p \overset{t}\lra E \overset{T_\pi E}\lra B\STop(2n),
\end{equation*}
which classifies $t^*T_\pi E$, will be used to detect classes in $B\STop(2n)$.

To evaluate $\cL_i(t^* T_\pi E)$ we evaluate the fibre integral
\begin{equation*}
  \int_\pi \cL_i(T_\pi E) \in H^{4i-2n}(BC_p ; \bZ[\tfrac{1}{(2i+1)!}])
\end{equation*}
in two ways: on the one hand, a suitable form of the family signature theorem (see \cite{RWSign} and the recollections below) gives a formula for this class in terms of the action of $C_p$ on cohomology of the fibres of $\pi$, and on the other hand it can be evaluated explicitly using equivariant localisation.  The upshot will be a formula which can be solved for $\cL_i(t^*T_\pi E)$ and used to prove Proposition~\ref{prop:NonLinRep}.

\subsection{The family signature theorem}

Let us summarise the family signature theorem for topological bundles as presented in \cite{RWSign} in the form that we will need it (in op.\ cit.\ it is developed in the generality of topological \emph{block} bundles, but we will only need the fibre bundle case).
  To an oriented fibre bundle $\pi: E \to B$ whose fibres are closed topological $2n$-manifolds, the family signature theorem that we will need can be stated as the formula
\begin{equation}\label{eq:1}
  \int_\pi \mathcal{L}_i(T_\pi E) = 2^{2i-n} \mathrm{ch}_{2i-n}(\Sign(\pi)) \in H^{4i-2n}(B;\bZ[\tfrac1{(2i+1)!}]),
\end{equation}
where $\int_\pi: H^{4i}(E) \to H^{4i-2n}(B)$ denotes the fibre integration homomorphism, $T_\pi E$ is the fiberwise tangent (micro)bundle, $\mathrm{ch}_{2i}$ is the Chern character, and $\Sign(\pi) \in K^0(B)$ is a complex $K$-theory class defined as follows (see \cite[Section 2.2]{RWSign} for more details: this agrees with the class called $\xi$ there).

\begin{definition}\label{defn:Sign}
  Writing $E_b = \pi^{-1}(b)$, the vector spaces $H^n(E_b;\bR)$ are the fibres of a real vector bundle $\mathcal{H}$ on $B$, on which we may choose an inner product.  The intersection pairing gives rise to a fiberwise linear automorphism $A$ of $\mathcal{H}$, defined by $\langle x, Ay\rangle = (x \cup y)([E_b])$ for $x, y \in H^n(E_b;\bR)$.
  For odd $n$ this operator is skew self-adjoint, and the endomorphism $J = A/\sqrt{A^* A}$ defines a complex structure on $\mathcal{H}$ and we set $\Sign(\pi) = (\mathcal{H},-J) - (\mathcal{H},J) \in K^0(B)$.

  For even $n$ the operator $A$ is self-adjoint so we obtain a decomposition $\mathcal{H} = \mathcal{H}^+ \oplus \mathcal{H}^-$ by eigenspaces of $A/\sqrt{A^* A}$.  In this case we set $\Sign(\pi) = (\mathcal{H}^+ - \mathcal{H}^-) \otimes \bC \in K^0(B)$.
\end{definition}

\begin{remark}
  If $B = BG$ for a finite group $G$ (we will use $G = C_p$) and the bundle arises from a continuous action of $G$ on a manifold $W$, then the bundle $\mathcal{H}$ on $BG$ arises from the real representation of $G$ given by the action on $H^n(W;\R)$.  Furthermore, the class $\Sign(\pi)$ admits an evident lift along the usual completion map $R(G) \to K^0(BG)$, defined using a $G$-invariant real inner product on $H^n(W;\R)$.  By abuse of notation we also denote this class
  \begin{equation*}
    \Sign(H^n(W;\R)) \in R(G),
  \end{equation*}
  but we emphasise that it depends on both the $G$-action on $H^n(W;\R)$ and the $G$-equivariant intersection form on this vector space.
\end{remark}
The fact that formula~\eqref{eq:1} holds rationally is essentially due to Meyer, but with the correct integral refinements of the Hirzebruch classes it also holds over $\bZ[\frac1{(2i+1)!}]$ as stated, see \cite[Section 6.4]{RWSign}.

\section{Cyclic group actions on topological manifolds}

In this section we state the main existence result about manifolds with cyclic group actions with certain properties, and explain how to deduce Proposition~\ref{prop:NonLinRep}.  The input data will be:

\begin{enumerate}[(i)]
\item an odd prime $p$ and an integer $n \geq 2$,
\item a $2n$-dimensional oriented orthogonal representation $V_s$ with $V_s^{C_p} = 0$,
\item a zero-dimensional virtual complex representation $\xi \in R(C_p)$  satisfying $\overline{\xi} = (-1)^n \xi$, where $\overline{\xi}$ is the complex conjugate.
\end{enumerate}

\begin{proposition}\label{prop:existence-of-enough-actions}
  Given the data above, there is a smooth compact $2n$-dimensional $C_p$-manifold $W$ and an equivariant immersion
    \begin{equation*}
      \phi: W \looparrowright V_s
    \end{equation*}
    such that
    \begin{enumerate}[(i)]
    \item $\phi$ restricts to a diffeomorphism in a neighborhood of $\phi^{-1}(0)$, 
    \item the boundary $\partial W$ is a homology sphere, with free $C_p$-action, and for $n \geq 3$ it is non-equivariantly diffeomorphic to $S^{2n-1}$, 
    \item\label{it:BdySphere} the resulting equivariant map
      \begin{equation*}
        \begin{aligned}
          \partial W \lra S(V_s)\\
          x \longmapsto \tfrac{\phi(x)}{|\phi(x)|}
        \end{aligned}
      \end{equation*}
      has degree 1, 
    \item  the closed $C_p$-manifold $\widehat{W} = W \cup_{\partial W} C(\partial W)$ (for $n = 2$ only a homology manifold) has
     \begin{equation*}
        \Sign(H^n(\widehat{W};\bR)) = 4 \cdot \xi \in R(C_p).
  \end{equation*}
    \end{enumerate}
\end{proposition}

The fact that the map in~(\ref{it:BdySphere}) has degree 1 is an easy consequence of the first two properties, but we prefer to state it for emphasis: it implies that the actions of $C_p$ on the spheres $\partial W$ and $S(V_s)$ have equal Euler classes in $H^{2n}(BC_p;\bZ)$.

\subsection{Equivariant localisation}
\label{sec:equiv-local}

Postponing the construction of these actions until Section~\ref{sec:construction-actions}, let us look at some consequences.  We first consider the case $n \geq 3$, where $\partial W$ is diffeomorphic to $S^{2n-1}$, so $\widehat{W}$ is a topological manifold.

The action of $C_p$ on $\widehat{W}$ has precisely two fixed points, namely the (unique) inverse image of the origin in $V_s$ and the cone point of $C(\partial W) \subset \widehat{W}$.  We denote the former by $s$ and the latter by $t$, and consider the fibre bundle
\begin{equation*}
 \pi :  (EC_p \times \widehat{W})/C_p = E \lra BC_p
\end{equation*}
and the two sections given by $s$ and $t$.  The fiberwise normal bundle of $s$ can be identified with the vector bundle $(EC_p \times V_s)/C_p \to BC_p$, which has Euler class
\begin{equation*}
  e(V_s) \neq 0 \in H^{2n}(BC_p;\bZ).
\end{equation*}
(It is nonzero as $V_s$ has no trivial subrepresentations, by assumption.)  The fiberwise normal bundle of $t$ is the disk bundle $(EC_p \times C(\partial W))/C_p \to BC_p$.  From Proposition \ref{prop:existence-of-enough-actions} (\ref{it:BdySphere}) we can deduce a fiberwise homotopy equivalence between the sphere bundles of $V_t$ and $V_s$, but the orientation inherited on $\partial W$ as the boundary of $W$ is the opposite of the orientation inherited as the boundary of the disk $C(\partial W)$.  When the bundle $\pi$ is given the orientation arising from the orientation of $V_s$, and when we temporarily write $V_t$ for the fiberwise normal bundle of the section $t$ with its orientation induced from that of $\pi$, we therefore have
\begin{equation*}
  e(V_t)  = -e(V_s) \in H^{2n}(BC_p;\bZ).
\end{equation*}

For any class $\alpha \in H^k(E)$  the equivariant localisation formula \cite[Theorem 3.1.6]{AlldayPuppe} therefore implies
\begin{equation}\label{eq:3}
  e(V_s) \smile \int_\pi \alpha = s^*\alpha - t^*\alpha \in H^k(BC_p).
\end{equation}
In more detail, it suffices to prove~\eqref{eq:3} in the localised ring $H^*(BC_p)[e(V_s)^{-1}]$ and the cited reference then implies that it suffices to consider relative cohomology classes $\alpha \in H^k(E, E \setminus (s(BC_p) \cup t(BC_p)))$ since the $H^*(BC_p)$-module $H^*(E \setminus (s(BC_p) \cup t(BC_p)))$ vanishes after inverting the Euler class. Excision gives
$$H^k(E, E \setminus (s(BC_p) \cup t(BC_p))) \cong H^{k}(V_s, V_s \setminus s(BC_p)) \oplus H^{k}(V_t, V_t \setminus t(BC_p)),$$
so such classes can in turn be written as $\alpha = \alpha_s \smile u_s + \alpha_t \smile u_t$ where $u_s \in H^{2n}(V_s, V_s \setminus s(BC_p))$ and $u_t \in H^{2n}(V_t, V_t \setminus t(BC_p))$ are the Thom classes associated to the orientation.  Fibre integrating a class of this form gives $\alpha_s + \alpha_t$, while the right hand side of~\eqref{eq:3} gives $\alpha_s \smile e(V_s) - \alpha_t \smile e(V_t) = e(V_s) \smile (\alpha_s + \alpha_t)$.

Assuming that $p > 2i+1$ and specialising to $\bZ[\tfrac{1}{(2i+1)!}]$-coefficients, we can set $\alpha = \cL_i(T_\pi E)$ in~\eqref{eq:3}.  Combining with the family signature formula, we obtain
\begin{align*}
  e(V_s) \smile 2^{2i-n}\ch_{2i-n}(4\xi) &= e(V_s) \smile 2^{2i-n}\ch_{2i-n}(\Sign(H^n(\widehat{W};\bR))) \\
  &= e(V_s) \smile \int_\pi \cL_i(T_\pi E)\\
  &= \cL_i(s^* T_\pi E) - \cL_i(t^* T_\pi E),
\end{align*}
which we may write as
\begin{align}\label{eq:4}
\begin{split}
  \cL_i(t^* T_\pi E) = \cL_i(V_s) &-  2^{2+2i-n} e(V_s)\smile\ch_{2i-n}(\xi)\\
 &   \quad\quad\quad\quad\quad\quad\quad\quad \in H^{4i}(BC_p; \bZ[\tfrac{1}{(2i+1)!}]).
\end{split}
\end{align}

The euclidean bundle $t^* T_\pi E$ with its induced orientation is classified by a map $BC_p \to B\STop(2n)$, which arises from group homomorphisms
\begin{equation}\label{eq:14}
  C_p \lra \Diff^+(S^{2n-1}) \lra \STop(S^{2n-1}) \lra \STop(D^{2n}) \lra \STop(2n),
\end{equation}
defined using a choice of diffeomorphism $\partial W \approx S^{2n-1}$.  The pullback of the Hirzebruch classes along these homomorphisms (after taking classifying spaces) is then given by the formula~\eqref{eq:4}.

\subsection{Proof of Proposition~\ref{prop:NonLinRep} for $n \geq 3$, assuming Proposition~\ref{prop:existence-of-enough-actions}}

Applying the above calculation with the input data $V_s$ given by the homomorphism $\rho: C_p \to \SO(2n)$ in Proposition~\ref{prop:NonLinRep}, we obtain group homorphisms~\eqref{eq:14} classifying $t^* T_\pi E$.  Now, choose any orientation \emph{reversing} diffeomorphism of $\partial W \approx S^{2n-1}$ and let $\phi: C_p \to \Diff^+(S^{2n-1}) \to \STop(2n)$ be obtained by conjugating~\eqref{eq:14} by that element.  Then the equation $e(V_s) = -e(t^*(T_\pi E))$ becomes $e(\phi) = e(\rho)$.  Furthermore, the formula~\eqref{eq:4} implies that $\cL_j(\phi) = \cL_j(t^*T_\pi E) = \cL_j(V_s)$ when $2j < n$.  It remains to see that $\cL_j(t^* T_\pi E))$ can take arbitrary values when $n \leq 2j \leq p-3$.

  As $2$ is a unit modulo $p$ (as $p$ is odd), and cup product with $e(V_s)$ defines an isomorphism $H^{4j-2n}(BC_p;\bF_p) \to H^{4j}(BC_p;\bF_p)$ for $n \leq 2j$, by the formula~\eqref{eq:4} it suffices to see that the classes
  \begin{equation*}
    \ch_{2j-n}(\xi) \in H^{4j-2n}(BC_p;\bF_p), \quad \lceil \tfrac n2\rceil \leq j \leq \tfrac{p-3}2
  \end{equation*}
  can take arbitrary values as $\xi \in R(C_p)$ range over zero-dimensional virtual representations with $\overline{\xi} = (-1)^n \xi$.  By Lemma~\ref{lem:xyz} below, we may first find some zero-dimensional $\xi \in R(C_p)$ achieving this.  Replacing $\xi$ by
\begin{equation*}
  m(\xi + (-1)^n\overline{\xi}) \in R(C_p)
\end{equation*}
for some $m \in \bZ$ representing $2^{-1} \in \bF_p$ ensures that $\overline{\xi} = (-1)^n \xi$ without changing $\ch_{2i-n}(\xi)$, since $\mathrm{ch}_{2i-n}(m( \xi + (-1)^n\overline{\xi})) = m(\mathrm{ch}_{2i-n}(\xi) + (-1)^n (-1)^{2i-n} \mathrm{ch}_{2i-n}(\xi)) = 2m \mathrm{ch}_{2i-n}(\xi) = \mathrm{ch}_{2i-n}(\xi)$.  This finishes the proof of Proposition~\ref{prop:NonLinRep}.\qed
  
\begin{lemma}\label{lem:xyz}
  The homomorphism
  \begin{align*}
    R(C_p) &\lra \prod_{j=0}^{p-1} H^{2j}(BC_p;\bF_p)\\
    \xi & \longmapsto (\ch_0(\xi), \dots, \ch_{p-1}(\xi))
  \end{align*}
  is surjective.
\end{lemma}
\begin{proof}
  In terms of the standard representation $\chi: C_p \to \UU(1)$ we have an integral basis $1, \chi, \dots, \chi^{p-1}$ of $R(C_p)$ and a mod $p$ basis $1, c, \dots, c^{p-1}$ of the codomain, with $c = c_1(\chi) \in H^2(BC_p;\bF_p)$.  As $\ch_j(\chi^r) = \frac{(r \cdot c)^j}{j!}$ the matrix for the homomorphism with respect to these bases is given by the product
  \begin{equation*}
    \begin{pmatrix}
      0! & &&\\
       & 1! &\\
      & & \ddots\\
      &&& (p-1)!
    \end{pmatrix}^{-1}
    \begin{pmatrix}
      0^0 & 1^0 & 2^0 & \cdots & (p-1)^0\\
      0^1 &1^1 & 2^1 &\cdots & (p-1)^1\\
      \vdots & \vdots & \vdots & \ddots & \vdots\\
      0^{p-1} & 1^{p-1} & 2^{p-1} & \cdots & (p-1)^{p-1}
    \end{pmatrix},
  \end{equation*}
  which is invertible modulo $p$ (the second factor is invertible by the Vandermonde formula).
\end{proof}

\subsection{The case $n=2$}
The only step of the above argument which does not immediately apply is that the homology 3-sphere $\partial W$ is likely not simply connected, so the homology manifold obtained by coning it off need not be a manifold.  Instead we appeal to \cite[Proposition 11.1C]{FreedmanQuinn}, which gives a compact contractible 4-manifold $D$ with $C_p$-action, whose boundary is equivariantly homeomorphic to $\partial W$, and where the fixed points $D^{C_p}$ consists of a single point $t \in D$.  The cited result applies because $\partial W$ is a homology 3-sphere and the action is free.  We can therefore re-define $\widehat{W} = W \cup_{\partial W} D$ in this case, and otherwise proceed as before.

\section{Construction of actions}
\label{sec:construction-actions}

Finally, we prove Proposition~\ref{prop:existence-of-enough-actions}.  The proof has two steps: first the $\bC$-linear data given by $\xi \in R(C_p)$ is upgraded to $\bZ$-linear data, then this $\bZ$-linear data is realised geometrically as the intersection form of a manifold.  This is essentially an instance of ``Wall realisation'' \cite[Theorem 5.8]{wallscm}, but let us first summarise the necessary definitions and results.

\subsection{$L$-theory}
\label{sec:l-theory}

We recall some details from the definition of Wall's (simple) algebraic $L$-theory groups $L_{2n}(\bZ[C_p])$, to set notation. We refer to \cite[Chapter 5]{wallscm} for full details. Elements 
 are represented by isomorphism classes of tuples $(P,\lambda,\mu)$ where
\begin{enumerate}[(i)]
\item $P$ is a finitely-generated free right $\bZ[C_p]$-module with a chosen basis,
\item $\lambda: P \times P \to \bZ[C_p]$ is a simple unimodular $(-1)^n$-hermitian form,
\item $\mu: P \to \bZ[C_p] / \bZ\{g - (-1)^n g^{-1} \mid g \in C_p\}$ is a quadratic refinement of $\lambda$.
\end{enumerate}
Unimodularity means that the adjoint $\bZ[C_p]$-linear homomorphism
\begin{equation*}
  A\lambda: P \lra \mathrm{Hom}_{\bZ[C_p]}(P,\bZ[C_p])
\end{equation*}
is an isomorphism; simple unimodularity means that the matrix of $A\lambda$ with respect to the given basis and its dual, represents the trivial element of the Whitehead group $\mathrm{Wh}(C_p)$.  We shall not need the explicit relations between these generators $[(P,\lambda,\mu)]$.

There is an associated $C_p$-invariant unimodular $(-1)^n$-symmetric form $b: P \times P \to \bZ$, given by the coefficient of $1 \in \bZ[C_p]$ in $\lambda$, equipped with a $C_p$-invariant quadratic refinement $q : P \to \bZ/(1 - (-1)^n)\bZ$ given by the coefficient of $1 \in \bZ[C_p]$ in $\mu$.

The multisignature is a homomorphism
\begin{equation}\label{eq:17x}
 \Sign : L_{2n}(\bZ[C_p]) \lra R(C_p) 
\end{equation}
defined in analogy to the construction in Definition \ref{defn:Sign}.  Explicitly, choose a $C_p$-invariant (positive definite) inner product on $P_\bR = \bR \otimes_\bZ P$ and write the $\bR$-bilinear extension $b_\bR$ of $b$ as $b_\bR(x,y) = \langle x, Ay\rangle$ for a $\bR[C_p]$-linear operator $A: P_\bR \to P_\bR$.  For odd $n$ this operator is skew self-adjoint, and the endomorphism $J = A/\sqrt{A^* A}$ defines a complex structure on $\mathcal{H}$ and we set $\Sign(P,\lambda,\mu) := (P_\bR,-J) - (P_\bR,J) \in R(C_p)$. For even $n$ the operator $A$ is self-adjoint so we obtain a decomposition $P_\bR = P_\bR^+ \oplus P_\bR^-$ by eigenspaces of $A/\sqrt{A^* A}$.  In this case we set $\Sign(P,\lambda,\mu) := (P_\bR^+ - P_\bR^-) \otimes \bC \in R(C_p)$.

The representation ring $R(C_p)$ is spanned as a $\bZ$-module by tensor powers of the 1-dimensional representation $\chi : C_p \to \UU(1)$  corresponding to the inclusion.
Bak \cite[Theorem 3]{BakLThy} and Wall \cite[Theorem 13A.4 (ii)]{wallscm} show that the image of $\Sign$ the span of $4(\chi^r + (-1)^n \chi^{-r})$ for $r = 0, \dots, \frac{p-1}2$.  Indeed, the representation ring has integral basis $\chi^{-(p-1)/2}, \dots, \chi^{(p-1)/2}$ and in the notation and language of the latter reference, the subgroup consisting of ``characters trivial on 1 (real or imaginary as appropriate)'' has basis $\chi^r + (-1)^n \chi^{-r}$ for $r = 1, \dots, \frac{p-1}2$.  In the splitting $L_{2n}(C_p) = L_{2n}(1) \oplus \tilde{L}_{2n}(C_p)$ mentioned in the cited Theorem, the image of the second summand in $R(C_p)$ therefore has basis $4(\chi^r + (-1)^n \chi^{-r})$ for $r = 1, \dots, \frac{p-1}2$.  The first summand $L_{2n}(1)$ is $\bZ/2$ when $n$ is odd and when $n$ is even it is $\bZ$, a generator of which is sent to $8 = 4(\chi^0 + \chi^0) \in R(C_p)$ .

Since the input datum $\xi \in R(C_p)$ is assumed to have virtual dimension 0 and to satisfy $\overline{\xi} = (-1)^n \xi$, we can write
\begin{equation*}
  \xi = \sum_{r=1}^{(p-1)/2} b_r\left((\chi^r + (-1)^n \chi^{-r}) - (\chi^0 + (-1)^n \chi^0)\right)
\end{equation*}
for some $b_r \in \bZ$, and hence there exists a $[(P,\lambda,\mu)] \in L_{2n}(\bZ[C_p])$ with $\Sign(P,\lambda,\mu) = 4\cdot \xi$.  For $n$ even, the fact that $\xi$ has virtual dimension 0 implies that the (usual) signature of the $\bR$-bilinear form $b_\bR$ is zero.  For $n$ odd, as the kernel of $\Sign$ is detected by the Arf invariant we can arrange that the quadratic refinement $q: P \to \bZ/2$ has Arf invariant zero.

\subsection{Wall realisation}
\label{sec:wall-realisation}

The final step in the proof of Proposition~\ref{prop:existence-of-enough-actions} is to realise $(P,\lambda, \mu)$ as the equivariant intersection form of a suitable $C_p$-manifold $W$.

To this end, let $X := S(V_s)/C_p$ denote the $(2n-1)$-dimensional lens space associated to the representation $V_s$, and let $(P, \lambda, \mu)$ be as above: namely it has multisignature $4 \cdot \xi$ and, if $n$ is odd, its associated quadratic form $q$ has Arf invariant 0.  We now claim that $(P,\lambda,\mu)$ can be represented geometrically by a $2n$-dimensional smooth cobordism $M$ equipped with an $n$-connected map
\begin{equation*}
  \phi:   M \lra [0,1] \times X,
\end{equation*}
which is also an immersion, restricts to a diffeomorphism $\phi_0: \partial_\mathrm{in} M \to \{0\} \times X$, and restricts to a simple homotopy equivalence $\partial_\mathrm{out}M \to [0,1] \times X$. By ``represented'' we mean that for any $M$ and $\phi$ with these properties, the $\bZ[C_p]$-module
\begin{equation}\label{eq:18}
  K_\phi := \pi_{n+1}(X, M) \cong H_{n+1}(X, M ; \bZ[C_p]) \cong H_n(M ;\bZ[C_p]),
\end{equation}
where the first isomorphism comes from the Hurewicz theorem and the second is because the universal cover of $X$ is $(2n-2)$-connected, can be given a geometrically defined intersection form $\lambda_\phi$ and self-intersection form $\mu_\phi$, and that it is possible to choose $M$ and $\phi$ such that there exists a $\bZ[C_p]$-linear isomorphisms $K_\phi \to P$ sending these to $\lambda$ and $\mu$. 

We define the geometric (self-)intersection forms $\lambda_\phi$ and $\mu_\phi$ associated to the $n$-connected immersion $\phi: M \to [0,1] \times X$ using \cite[Theorem 5.2]{wallscm}. To apply that theorem we must provide a map $K_\phi \to \pi_n(\mathrm{Fr}_n(TM))$ of right $\bZ[C_p]$-modules, to the $n$th homotopy group of the bundle of $n$-frames in $M$, so that by Smale--Hirsch theory elements of $K_\phi$ represent immersed based spheres in $M$. We do this as
$$\pi_{n+1}(X, M) \overset{\sim}\leftarrow \pi_{n+1}(\mathrm{Fr}_{2n}(\bR \oplus TX), \mathrm{Fr}_{2n}(TM)) \overset{\partial}\to \pi_n(\mathrm{Fr}_{2n}(TM)) \to \pi_n(\mathrm{Fr}_{n}(TM)),$$
where the left-hand map is an isomorphism because the evident square involving these four spaces is homotopy cartesian. (As the map just constructed factors through the $2n$th frame bundle, the corresponding immersions will have trivial normal bundle, and hence the ``normal Euler number'' term in \cite[Theorem 5.2 (iii)]{wallscm} will vanish.)

\begin{remark}
Constructing $\lambda_\phi$ and $\mu_\phi$ from $\phi: M \to [0,1] \times X$ in this way is essentially identical to the definition in \cite[Chapter 5]{wallscm} of the quadratic form on the surgery kernel, except that we have replaced the notion of ``degree 1 normal map'' in op.\ cit.,\ which is extra data on a map, with the geometric \emph{condition} that $\phi$ be an immersion.  Readers familiar with the discussion in \cite[Chapter 5]{wallscm} will see how to extract a ``degree 1 normal map'' $M \to [0,1] \times X$ from the immersion $\phi$: first extract a bundle map $TM \to \bR \oplus TX$ from the derivative $D\phi$, then homotope $\phi: M \to [0,1] \times X$ relative to $\partial_\mathrm{in} M$ to a map sending $\partial_{\mathrm{out}}$ to $\{1\} \times X$ and lift to a covering homotopy of bundle maps $TM \to \bR \oplus TX$.
\end{remark}

The cobordism $M$ and its immersion into $[0,1] \times X$ is constructed by attaching $n$-handles to the trivial cobordism $[0,\tfrac{1}{2}] \times X$, one for each of the given basis elements in the free $\bZ[C_p]$-module $P$.  As explained in the proof of \cite[Theorem 5.8]{wallscm}, such immersed handle attachments can realise any $(\lambda,\mu)$ as intersection and self-intersection forms, and the fact that $\lambda$ is unimodular implies that the restriction $\phi\vert_{\partial_\mathrm{out} M}$ induces an isomorphism on homology with $\bZ[C_p]$-coefficients.  This restriction is also $(n-1)$-connected, so for $n \geq 3$ it is a homotopy equivalence by the Hurewicz theorem.

We can model the universal cover of $[0,1] \times X$ as the map
\begin{align*}
  2 D(V_s) \setminus \mathring{D}(V_s) & \lra [0,1] \times X\\
  v \quad\quad& \longmapsto \big(\|v\|-1,\big[\tfrac{v}{\|v\|}\big]\big),
\end{align*}
and let $\widetilde{M} = (2D(V_s) \setminus \mathring{D}(V_s)) \times_{([0,1] \times X)} M$ be the pullback of $\phi$.  This $C_p$-cobordism $\widetilde{M}$ inherits an equivariant immersion
\begin{equation*}
  \widetilde{\phi}: \widetilde{M} \lra 2D(V_s) \setminus \mathring{D}(V_s)
\end{equation*}
which restricts to an equivariant diffeomorphism $\partial_\mathrm{in} \widetilde{M} \to S(V_s)$.  For $n \geq 3$ the restriction to $\partial_{\mathrm{out}} \widetilde{M}$ is a homotopy equivalence, while for $n=2$ it is only an integral homology equivalence.  The glued $C_p$-manifold
\begin{equation*}
  W = D(V_s) \cup_{S(V_s)} \widetilde{M}
\end{equation*}
inherits an equivariant immersion into $V_s$ which is a diffeomorphism over $D(V_s)$, and the surgery kernel~\eqref{eq:18} may be identified with $H_n(W;\bZ)$, where the latter is given the $\bZ[C_p]$-module structure arising from the geometric action.  Indeed, the surgery kernel~\eqref{eq:18} maps isomorphically to the relative homology
\begin{equation*}
  H_n(M,\partial_\mathrm{in} M;\bZ[C_p]) \cong H_n(\widetilde{M},\partial_\mathrm{in}\widetilde{M}) \stackrel{\cong}{\lra} H_n(W,D(V_s);\bZ) \stackrel{\cong}{\longleftarrow} H_n(W;\bZ),
\end{equation*}
where the first arrow is an isomorphism by excision and the second is an isomorphism by the long exact sequence of the pair.  The $\bZ[C_p]$-module structure on the surgery kernel corresponds to the action on $H_n(W;\bZ)$ induced from the action on $W$.  Furthermore, the $\bZ$-bilinear intersection form on $H_n(W;\bZ)$ is invariant under the $C_p$ action, and is identified with the form previously denoted $b: P \times P \to \bZ$, the coefficient of $1 \in \bZ[C_p]$ in $\lambda$.  Similarly, the quadratic refinement $q: P \to \bZ/(1 - (-1)^n) \bZ$ is identified with the self-intersection form on $H_n(W;\bZ)$, defined using the framing arising from the immersion into $V_s$.

It only remains to prove that for $n \geq 3$ the homotopy sphere $\partial W$ is (non-equivariantly) diffeomorphic to $S^{2n-1}$.  The framing of $W$ already shows that it lies in the subgroup $\mathrm{bP}_{2n} \subset \Theta_{2n-1}$ in the Kervaire--Milnor classification \cite{kervairemilnor}.  To see that $\partial W$ represents $0 \in \mathrm{bP}_{2n}$ for even $n$, we have to determine the signature of the bounding framed manifold, but we have arranged that $b: P \times P \to \bZ$ has signature 0.  Similarly, for $n$ odd we have to examine the Arf invariant of the bounding framed manifold, but we have arranged that $q: P \to \bZ/2\bZ$ has Arf invariant 0.\qed

\section{Pr\"ufer group actions on spheres}
\label{sec:surg-exact-sequ}

We have finished the proof of all the results announced in the introduction, but it still remains to explain why the homomorphisms $C_p \to \Diff^+(S^{2n-1})$ extend to the Pr\"ufer group $C_{p^\infty}$, as claimed in Proposition~\ref{prop:NonLinRep} in the case $2n-1 \geq 5$.  In this section our presentation will assume familiarity with the (smooth) \emph{surgery exact sequence} as developed in \cite[Chapter 10]{wallscm}, and we first recast the construction from Section~\ref{sec:construction-actions} in that context.

\subsection{The surgery exact sequence}
\label{sec:surg-exact-sequ-1}

\newcommand{\lensspace}{X}  

In the setup from Section~\ref{sec:construction-actions}, the smooth surgery exact sequence for the lens space $\lensspace = S(V_s)/C_p$ looks like
\begin{equation*}
  L_{2n}(\bZ[C_p]) \overset{\partial}\lra \mathcal{S}(\lensspace) \overset{\eta}\lra [\lensspace,G/O] \overset{\sigma}\lra L_{2n-1}(\bZ[C_p])
\end{equation*}
whose outer terms are the (simple) quadratic $L$-theory
of the group ring $\bZ[C_p]$, whose third term is the so-called normal invariants, and whose second term $\mathcal{S}(\lensspace)$ is the \emph{structure set}, i.e.\ the set of $s$-cobordism classes of pairs $(M,f)$ where $M$ is a smooth manifold and $f: M \to \lensspace$ is a simple homotopy equivalence.  Exactness at $\mathcal{S}(\lensspace)$ means that the abelian group $L_{2n}(\bZ[C_p])$ acts on $\mathcal{S}(\lensspace)$ and that $\eta$ induces an injection from the set of orbits.  (In fact this injection is also a surjection because $L_{2n-1}(\bZ[C_p])=0$ by \cite{BakOdd}, but we shall not use that.)

The identity map $\iota: \lensspace \to \lensspace$ represents an element $\iota \in \mathcal{S}(\lensspace)$ and the homotopy equivalence $(\partial W)/C_p \to S(V_s)/C_p$ induced from Proposition~\ref{prop:existence-of-enough-actions} gives another element.  By construction these elements are normally bordant, and hence in the same orbit under the action of $L_{2n}(\bZ[C_p])$.

\subsection{Pr\"ufer group actions}
\label{sec:prufer-group-actions}

Choose an extension of the given linear representation $V_s$ to a representation of the Pr\"ufer group $C_{p^\infty}$, and write $\lensspace_k = S(V_s)/C_{p^k}$ for all $k \geq 1$.  Then we have an inverse system of sets
\begin{equation}\label{eq:21}
  \mathcal{S}(\lensspace_1) \longleftarrow \mathcal{S}(\lensspace_2) \longleftarrow \cdots,
\end{equation}
where the maps are given by passing to $p$-fold covers, and a corresponding inverse system of surgery exact sequences
\begin{equation*}
  L_{2n}(\bZ[C_{p^k}]) \overset{\partial}\lra \mathcal{S}(\lensspace_k) \overset{\eta}\lra [\lensspace_k,G/O] \overset{\sigma}\lra L_{2n-1}(\bZ[C_{p^k}]).
\end{equation*}
This inverse system is clear if the $L$-groups are described geometrically in terms of surgery problems, as in \cite[Chapter 9]{wallscm}, where the finite covering map $X_{k-1} \to X_{k}$ induces contravariant maps on $L$-groups by pullback of surgery problems, and on normal invariants by precomposition. Under the identification of the geometrically defined $L$-groups with the algebraic ones \cite[Corollary 9.4.1]{wallscm}, pullback of surgery problems yields maps 
$$\tau_k: L_{2n}(\bZ[C_{p^{k}}]) \lra L_{2n}(\bZ[C_{p^{k-1}}])$$
known as transfer maps. Unravelling definitions, this map sends the class represented by $(P, \lambda, \mu)$ to the class represented by $(P', \lambda', \mu')$ where $P'$ is $P$ considered as a right $\bZ[C_{p^{k-1}}]$-module by restriction along $\bZ[C_{p^{k-1}}] \to \bZ[C_{p^{k}}]$, and $\lambda'$ and $\mu'$ are obtained from $\lambda$ and $\mu$ by postcomposition with (the maps induced by) the trace map $\mathrm{Tr} : \bZ[C_{p^{k}}] \to \bZ[C_{p^{k-1}}]$ sending $g \in C_{p^{k-1}} \leq C_{p^{k}}$ to $g$, and $g \in C_{p^{k}} \setminus C_{p^{k-1}}$ to zero. In particular, the $C_{p^{k-1}}$-invariant unimodular $(-1)^n$-symmetric form $b' : P \times P \to \bZ$ extracted from $\lambda'$ agrees with the form $b$ extracted from $\lambda$. This shows that the multisignature map intertwines $\tau_k$ and the restriction map $R(C_{p^k}) \to R(C_{p^{k-1}})$ on complex representation rings.

The identity maps $(\iota_k: \lensspace_k \to \lensspace_k) \in \mathcal{S}(\lensspace_k)$ define a point in the inverse limit of~\eqref{eq:21}, and we look at the subsets
\begin{equation*}
  T_k := \{x \in \mathcal{S}(\lensspace_k) \mid x \sim \iota_k \},
\end{equation*}
where we have temporarily written $\sim$ to denote the equivalence relation that representatives are normally bordant by a bordism whose universal cover has vanishing signature (for $n$ even) or Arf invariant (for $n$ odd).  By the general surgery theory, two elements of the structure set have normally bordant representatives if and only if they are in the same orbit for the action of the $L$-groups, and this finer relation $\sim$ holds if and only if the representatives are in the same orbit for the action of the subgroups
\begin{equation}\label{eq:20}
  L'_{2n}(\bZ[C_{p^{k}}]) := \mathrm{Ker}\big(L_{2n}(\bZ[C_{p^{k}}]) \xrightarrow{\tau_1 \circ \dots \circ \tau_k} L_{2n}(\bZ) \big).
\end{equation}

The subsets $T_k \subset \mathcal{S}(\lensspace_k)$ also form an inverse system and in this language and notation, our proof of Propositions~\ref{prop:NonLinRep} and~\ref{prop:existence-of-enough-actions} amounted to constructing an element $((\partial W)/C_p \to \lensspace_1) \in T_1$ with the required properties.  The key observation is now the following.
\begin{lemma}
  The maps in the inverse system~\eqref{eq:21} restrict to surjections $T_{k} \to T_{k-1}$.
\end{lemma}
\begin{proof}
By exactness of the surgery exact sequence it suffices to see that the restriction $\tau_k': L'_{2n}(\bZ [C_{p^{k}}]) \to L'_{2n}(\bZ[C_{p^{k-1}}])$ of the transfer maps to the subgroups~\eqref{eq:20} is surjective.
  
The multisignature identifies $L'_{2n}(\bZ[C_{p^{k}}])$ with a subgroup of the kernel of the restriction homomorphism $R(C_{p^k}) \to R(1) = \bZ$, namely 4 times the appropriate eigenspace for complex conjugation. Surjectivity then follows from the evident surjectivity of the restriction homomorphism $R(C_{p^k}) \to R(C_{p^{k-1}})$.
\end{proof}

The lemma implies that the element $((\partial W)/C_p \to \lensspace_1) \in T_1 \subset \mathcal{S}(\lensspace_1)$ admits a lift to the inverse limit of the subsystem $T_1 \leftarrow T_2 \leftarrow \cdots$.
The homomorphism $\phi: C_p \to \Diff(S^{2n-1})$ arose from the action on $\partial W$ upon choosing a diffeomorphism $W \approx S^{2n-1}$, and the compatible lifts of $((\partial W)/C_p \to \lensspace_1)$ to $T_k \subset \mathcal{S}(\lensspace_k)$ give compatible extensions of this $\phi$ to  $C_{p^k}$.

\bibliographystyle{amsalpha}
\bibliography{biblio}

\end{document}